\pdfoutput=1
\documentclass[reqno, 11pt]{amsart}
\usepackage{amsthm, amsmath,amssymb,microtype}
\usepackage[T1]{fontenc}

\usepackage[pdfstartview=FitV, 
pagebackref=false]{hyperref}

\usepackage[margin=1in]{geometry}

\newtheorem{theorem}{Theorem}

\newtheorem{lemma}[theorem]{Lemma}

\newtheorem*{question*}{Question}

\theoremstyle{definition}

\newtheorem*{definition*}{Definition}

\theoremstyle{remark}

\newcommand{\abs}[1]{\left\lvert#1\right\rvert}
\newcommand{\norm}[1]{\left\lVert#1\right\rVert}

\newcommand{\ceil}[1]{\left\lceil #1 \right\rceil}

\newcommand{\RR}{\mathbb{R}}

\newcommand{\NN}{\mathbb{N}}

\newcommand{\cH}{\mathcal{H}}
\newcommand{\cM}{\mathcal{M}}
\newcommand{\cW}{\mathcal{W}}

\newcommand{\x}{\times}
\newcommand{\inj}{\mathrm{inj}}
\newcommand{\simple}{\mathrm{simple}}
\DeclareMathOperator{\Adm}{Adm}

\author{L\'aszl\'o Mikl\'os Lov\'asz}
\address{Department of Mathematics\\ MIT\\ Cambridge, MA 02139, United States}
\email{lmlovasz@math.mit.edu}

\author{Yufei Zhao}
\address{Mathematical Institute, University of Oxford, Oxford OX2 6GG, United Kingdom}
\email{yufei.zhao@maths.ox.ac.uk}
\thanks{Y.~Zhao was supported by a Microsoft Research PhD Fellowship.}

\title{On derivatives of graphon parameters}

\begin{document}

\begin{abstract}
We give a short elementary proof of the main theorem in the paper
``Differential calculus on graphon space'' by Diao et al.~(2015)~\cite{DGKR}, which says
that any graphon parameters whose $(N+1)$-th derivatives all vanish
must be a linear combination of homomorphism densities $t(H,
-)$ over graphs $H$ on at most $N$ edges.
\end{abstract}

\maketitle

Let $\cW \subset L^\infty([0,1]^2,\RR)$ denote the set of bounded symmetric measurable functions $f \colon [0,1]^2 \to \RR$ (here symmetric means $f(x,y) = f(y,x)$ for all $x,y$). Let $\cW_{[0,1]} \subset \cW$ denote those functions in $\cW$ taking values in $[0,1]$. Such functions, known as \emph{graphons}, are central to the theory of graph limits \cite{Lov}, an exciting and active research area giving an analytic perspective towards graph theory.

In \cite{DGKR}, the authors systematically study the local structure of differentiable graphon parameters. They develop the theory of consistency constraints for multilinear functionals on graphon space, and as a consequence, obtain the result (Theorem~\ref{thm:DGKR} below) that is the graphon analog of the following basic fact from calculus: the set of functions whose $(N+1)$-th derivatives all vanish identically is precisely the set of polynomials of degree at most $N$. For graphons, homomorphism densities $t(H, -)$ play the role of monomials: they generate a ring of smooth functions that separate points and they have the property of vanishing higher derivatives as in Theorem~\ref{thm:DGKR}. In this short note, we follow a more direct route to prove their result. Our proof avoids the technicalities of the approach in \cite{DGKR}.

We begin with some definitions. The space $\cW$ is equipped with the \emph{cut norm}
\[
\norm{f}_\square := \sup_{\text{measurable }S, T \subseteq [0,1]} \abs{
  \int_{S \x T} f(x,y) \, dxdy }.
\]
Given $g \in \cW$, and a measure-preserving map $\phi \colon [0,1] \to [0,1]$, we define $g^\phi(x,y) := g(\phi(x), \phi(y))$.
The \emph{cut distance} on $\cW$ is defined by $\delta_\square(f,g) := \inf_{\phi} \norm{f - g^\phi}_\square$
where $\phi$ ranges over all such measure-preserving maps. Let $\sim$
denote the equivalence relations in $\cW$ defined by
$f \sim g \Leftrightarrow \delta_\square(f,g) = 0$. It is known that
$(\cW_{[0,1]}/\sim, \delta_\square)$ is a compact metric space~\cite{LS07}.

Functions $F \colon \cW_{[0,1]} / \sim \to \RR$ are called \emph{class
  functions} (we import this terminology from \cite{DGKR}; the term
\emph{graphon parameter} is also used in the literature). Class functions that are
continuous with respect to the cut distance play an important role in
graph parameter/property testing~\cite{BCLSV1,LS10}.

Define the
\emph{admissible directions at $f \in \cW_{[0,1]}$} as
\[
\Adm(f) := \{ g \in \cW : f + \epsilon g \in \cW_{[0,1]} \text{ for
  some $\epsilon > 0$} \}.
\]
The \emph{G\^ateaux derivative of $F$ at $f \in \cW_{[0,1]}$ in the
direction $g \in \Adm(f)$} is defined by (if it exists)
\[
dF(f;g) := \lim_{\lambda \to 0^+} \frac{1}{\lambda} (F(f + \lambda
g) - F(f)).
\]
\emph{Higher mixed G\^ateaux derivatives} are defined iteratively: $d^{N+1}F(f;
g_1, \dots, g_{N+1})$ is defined to be the G\^ateaux derivative of
$d^N F(-; g_1, \dots, g_N)$ at $f$ in the direction $g_{N+1}$, if this limit exists.

Let $\cH_n$ denote the isomorphism classes of multi-graphs with $n$
edges, no isolated vertices, and no self-loops but possible
multi-edges. Also let $\cH_{\le n} := \bigcup_{j \le n} \cH_j$ and
$\cH := \bigcup_{j \in \NN} \cH_j$.

For any $H \in \cH$, and any $f \in \cW$, we define the homomorphism density
\[
t(H, f) := \int_{[0,1]^{V(H)}} \prod_{ij \in E(H)} f(x_i, x_j)
\prod_{i \in V(H)} dx_i,
\]
where $E(H)$ is the multi-set of edges of $H$. For example, when $H$ consists of two vertices
and two parallel edges between them, $t(H, W) = \int_{[0,1]^2} W(x,y)^2 \,
dxdy$.

Here is the main result of \cite{DGKR}.

\begin{theorem}[Diao, Guillot, Khare, Rajaratnam~{\cite[Theorem 1.4]{DGKR}}] \label{thm:DGKR}
  Let $F \colon \cW_{[0,1]} \to \RR$ be a class function which is
  continuous with respect to the $L^1$ norm and $N+1$ times
  G\^ateaux differentiable for some $N \ge 0$. Then $F$ satisfies
  \[
  d^{N+1} F(f; g_1, \dots, g_{N+1}) = 0, \qquad
  \forall f \in \cW_{[0,1]}, \ g_1, \dots, g_{N+1} \in \Adm(f),
  \]
  if and only if there exist constants $c_H$ such that
  \begin{equation} \label{eq:F=sum-t}
  F(f) = \sum_{H \in \cH_{\le N}} c_H t(H, f).
  \end{equation}
  Moreover, the constants $c_H$ are unique. If in addition $F$ is
  continuous with respect to the cut norm, then $c_H = 0$ if $H \in
  \cH_{\le N}$ is not a simple graph.
\end{theorem}

The ``if'' direction is simple. From the definition, we can see that $t(H, f +
\lambda_1g_1 + \dots + \lambda_{N+1} g_{N+1})$ expands into a polynomial
in $\lambda_1, \dots, \lambda_{N+1}$ of total degree at most $|E(H)| \le N$,
which clearly implies that its derivative with respect to $d\lambda_1 d\lambda_2 \dots
d\lambda_{N+1}$ vanishes identically. Thus any $F$ of the form
\eqref{eq:F=sum-t} satisfies $d^{N+1}F
\equiv 0$ (and is $L^1$-continuous).

For the ``only if'' direction, we first give a sketch. When the domain
of $F$ is restricted to graphons that correspond to edge-weighted graphs on $n$ vertices, $F$ is simply a
function on $\binom{n}{2}$ real variables. So the vanishing of its $(N+1)$-th order
derivatives implies that it is a polynomial of degree at
most $N$. From these polynomials we can recover the coefficients of
$t(H, -)$. Weighted graphs on finitely many vertices correspond
to graphons that are step functions, and they are dense in
$\cW_{[0,1]}$ with respect to the $L^1$ norm, so the claim follows by continuity.

Now come the details. Let $\cM_n$ denote the set of symmetric $n
\x n$ matrices $a = (a_{i,j})$ with zeros on the diagonal ($a_{i,i} = 0$), and let $\cM_{n,[0,1]}
\subset \cM_n$ be the matrices with entries in $[0,1]$. We view
elements of $\cM_n$ as edge-weighted complete graphs on $n$ labeled
vertices. For $a, b \in \cM_n$, we write $a \sim b$ if $a$ can be
obtained from $b$ by a permutation of the vertex labels. We define
class functions and G\^ateaux derivatives for $\cM_n$ analogously to
how they are defined for $\cW$.
Write $[n]:=\{1,\dots,n\}$.
For any $a \in \cM_n$ and $H \in \cH$ (assume that $V(H) = \{1, \dots,
|V(H)|\}$), define
\begin{equation} \label{eq:mat-t}
t(H, a) = \frac{1}{n^{|V(H)|}} \sum_{v_1, \dots, v_{|V(H)|} \in [n]}
\prod_{ij \in E(H)} a_{v_i, v_j}.
\end{equation}
There is a natural embedding
$\cM_n \hookrightarrow \cW$, identifying $a \in \cM_n$ with $f_a \in
\cW$ given by $f_a(x,y) = a_{\ceil{nx}, \ceil{ny}}$ (and $f_a(x,y) = 0$
if $x$ or $y$ is $0$). All previous notions are consistent with the
identification.

Note that $t(H, a)$ is a degree $|E(H)|$ polynomial in $a_{i,j}$, $1 \le i
< j \le n$ (recall that $a$ was symmetric, so $a_{i,j} = a_{j,i}$). Write $(n)_k := n(n-1)\cdots (n-k+1)$ and define
\begin{equation} \label{eq:mat-t-inj}
t^{\inj}(H, a) = \frac{1}{(n)_{|V(H)|}} \sum_{\text{distinct } v_1, \dots, v_{|V(H)|} \in [n]}
\prod_{ij \in E(H)} a_{v_i, v_j}.
\end{equation}
For each fixed $H$ and $n \ge |V(H)|$, $t(H, a)$ equals a nonzero multiple of $t^\inj(H,a)$ plus a linear combination
of various $t^\inj(H', a)$ with $|E(H')| = |E(H)|$ and $|V(H')| <
|V(H)|$ (essentially recording the different ways that $v_1, \dots,
v_{|V(H)|}$ can fail to be distinct in the summation for $t(H,a)$). It
follows that $(t^{\inj}(H, -) : H \in \cH_N)$ can be transformed into $(t(H, -): H \in
\cH_N)$ via a lower triangular matrix with positive diagonal entries (when $\cH_N$ is sorted by the
number of vertices), and vice versa (since such matrices are invertible).

Let $\cH_d^{(n)}$ consist of those $H \in \cH_d$ with at most $n$ vertices. The main observation we need to make is the following lemma:
\begin{lemma} \label{lem:poly=>t-inj}
If a class function $F \colon \cM_{n,[0,1]} \to \RR$ is a homogeneous polynomial
of degree $d$, then we can write $F = \sum_{H \in \cH^{(n)}_d} c_H t^\inj(H, -)$ for
some $c_H \in \RR$, in a unique way.
\end{lemma}

\begin{proof}
  Since $F$ is a class function, the coefficient of the monomial
  $a_{i_1,j_1}\dots a_{i_d,j_d}$ is equal to the coefficient of
  $a_{\sigma(i_1),\sigma (j_1)}\dots a_{\sigma(i_d),\sigma(j_d)}$ for
  all permutations $\sigma$ of $[n]$. Observe that the
  polynomial $\sum_{\sigma \in S_n} a_{\sigma(i_1),\sigma (j_1)}\dots
  a_{\sigma(i_d),\sigma(j_d)}$ is a multiple of $t^\inj(H, a)$ for
  the multigraph $H$ whose multi-set of edges is given by $E(H) = \{i_1j_1, \dots, i_dj_d\}$. For
distinct $H$ and $H'$, the set of monomials that appear in $t^\inj(H, a)$ and
$t^\inj(H', a)$ are disjoint.  Thus, we have a direct correspondence between linear combinations of $t^{\inj}(H,-)$ for $H \in \cH_d^{(n)}$ and polynomials of degree $d$.
\end{proof}

In particular, this lemma implies the following:

\begin{lemma} \label{lem:indep}
  The elements of $\{t(H, -) : H \in \cH_{\le N}\}$ are linearly independent
  as functions on $\cM_{n,[0,1]}$ whenever $n \ge 2N$.
\end{lemma}

\begin{proof}
If $n \ge 2N$, then any graph $H$ with at most $N$ edges and no isolated vertices has at most $2N$ vertices. Thus the polynomials $\{t^\inj(H, -)$, $H \in \cH_{\le N}\}$ are
linearly independent. By the linear
  relations between $\{t(H,-)\}$ and $\{t^\inj(H,-)\}$, it
  follows that $\{t(H,-): H \in \cH_{\le N}\}$ is linearly independent as well.
\end{proof}

\begin{lemma}
  \label{lem:Mn}
  If $F \colon \cM_{n,[0,1]} \to \RR$ is a class
function whose $(N+1)$-th derivatives vanish everywhere, then $F =
\sum_{H \in \cH_{\le N}} c_H t(H, -)$ for some $c_H \in \RR$. If $n \ge 2N$, the values $c_H$ are uniquely determined.
\end{lemma}

\begin{proof} Note that $\cM_{n,[0,1]}$ is a subset of a finite dimensional vector space, which means $F$ is a function of $\binom{n}{2}$ real variables, and its G\^ateaux derivatives are just the usual partial derivatives.
So if the $(N+1)$-th derivatives of $F$ all vanish, then $F$ must
be a polynomial of degree at most $N$. By Lemma~\ref{lem:poly=>t-inj}, $F$ lies in the span of
$t^\inj(H, -)$, $H \in \cH_{\le N}$, and hence it lies in the span of
$t(H, -)$, $H \in \cH_{\le N}$. By Lemma \ref{lem:indep}, if $n \ge 2N$, the functions $t(H,-)$ are linearly independent, so the values $c_H$ are unique.
\end{proof}

Now we prove the ``only if'' direction of Theorem~\ref{thm:DGKR}. By
embedding $\cM_n \hookrightarrow \cW$, the hypothesis $d^{N+1}F \equiv
0$ on $\cM_{n,[0,1]}$ implies, by Lemma~\ref{lem:Mn},
that $F = \sum_{H \in \cH_{\le N}} c_{H}^{(n)} t(H, -)$ on
$\cM_{n,[0,1]}$ for some $c_H^{(n)}$, uniquely if $n \ge 2N$. For any $m, n \ge 2N$ with $m/n \in \NN$,
the image of $\cM_n$ in $\cW$ is contained in the image of $\cM_m$. Since
$F = \sum_{H \in \cH_{\le N}} c_{H}^{(m)} t(H, -)$ on $\cM_m$,
restricting to $\cM_n$, we see that
$c_H^{(n)} = c_H^{(m)}$ for all $H \in \cH_{\le N}$. It then follows that for any $n,n'
\ge 2N$, $c_H^{(n)} = c_H^{(nn')} = c_H^{(n')}$, so there is some
$c_H$ so that $c_H^{(n)} = c_H$ for all $n \ge 2N$.

It follows that $F = \sum_{H \in \cH_{\le N}} c_H t(H, -)$ on
$\bigcup_{n \in \NN} \cM_{n,[0,1]}$, whose image is dense in $\cW_{[0,1]}$
with respect to the $L^1$ norm. As both sides of the equation are continuous with
respect to the $L^1$ norm, the equality holds in all of $\cW_{[0,1]}$. The uniqueness of the constants $c_H$ follows from Lemma~\ref{lem:Mn}.

The proof of the final claim in Theorem~\ref{thm:DGKR} is reproduced here from \cite{DGKR} for completeness. Suppose
$F$ is continuous with respect to the cut norm. Then
\begin{equation} \label{eq:simplify}
F(f) = \sum_{H \in \cH_{\le N}} c_Ht(H^{\simple}, f)
\end{equation}
where $H^\simple$ is the simple graph obtained from $H$ by replacing any
multi-edge by a single edge between the same pair of vertices. Indeed,
\eqref{eq:simplify}
holds for $\{0,1\}$-valued $f$ since $t(H^{\simple}, f) = t(H, f)$
for all $\{0,1\}$-valued $f$. Since the set of $\{0,1\}$-valued graphons is dense in
$\cW_{[0,1]}$ with respect to cut distance, and both sides of
\eqref{eq:simplify} are continuous in $f$ with respect to cut distance,
\eqref{eq:simplify} holds on all of $\cW_{[0,1]}$. Thus only simple graphs
are needed in the summation for $F$.


\begin{thebibliography}{9}

\bibitem{BCLSV1}
  C.~Borgs, J.~T. Chayes, L.~Lov\'asz, V.~T. S{{\'o}}s, and K.~Vesztergombi,
  \emph{Convergent sequences of dense graphs. {I}. {S}ubgraph frequencies,
    metric properties and testing},
  Adv. Math. \textbf{219} (2008), 1801--1851.

\bibitem{DGKR}
  P. Diao, D. Guillot, A. Khare, and B. Rajaratnam,
  \emph{Differential calculus on graphon space,}
  J. Combin. Theory Ser. A \textbf{133} (2015), 183--227.

\bibitem{Lov}
  L.~Lov\'asz, \emph{Large networks and graph limits}, American Mathematical
  Society Colloquium Publications, vol.~60, American Mathematical Society,
  Providence, RI, 2012.

\bibitem{LS07}
  L.~Lov\'asz and B.~Szegedy, \emph{Szemer{\'e}di's lemma for the analyst},
  Geom. Funct. Anal. \textbf{17} (2007), 252--270.

\bibitem{LS10}
  L.~Lov\'asz and B.~Szegedy,
  \emph{Testing properties of graphs and functions},
  Israel J. Math. \textbf{178} (2010), 113--156.


\end{thebibliography}
\end{document}